\documentclass[10pt]{elsarticle}
\usepackage[cp1251]{inputenc}
\usepackage[english]{babel}
\usepackage{amsmath}
\usepackage{amssymb}
\usepackage{amsfonts}
\usepackage{graphicx}
\usepackage{floatflt,epsfig}
\usepackage{lineno,hyperref}
\usepackage{enumerate}
\usepackage{colortbl}
\usepackage{array,tabularx,tabulary,booktabs}
\usepackage{longtable}
\usepackage{multirow}
\usepackage{wrapfig}
\usepackage{subcaption}
\usepackage[table]{xcolor}
\newcolumntype{^}{>{\currentrowstyle}}

\journal{European Journal of Combinatorics}
\newtheorem{lemma}{Lemma}

\newtheorem{theorem}{Theorem}
\newtheorem{corollary}{Corollary}
\newtheorem{proposition}{Proposition}
\newtheorem{problem}{Problem}
\newtheorem{example}{Example}
\newcommand{\proof}{\medskip\noindent{\bf Proof.~}}
\begin{document}
\renewcommand{\abstractname}{Abstract}
\renewcommand{\refname}{References}
\renewcommand{\tablename}{Figure.}
\renewcommand{\arraystretch}{0.9}
\thispagestyle{empty}
\sloppy

\begin{frontmatter}
\title{Deza graphs: a survey and new results}

\author[01]{Sergey Goryainov}
\ead{sergey.goryainov3@gmail.com}

\author[01,02]{Leonid~V.~Shalaginov}
\ead{44sh@mail.ru}

\address[01] {Chelyabinsk State University, Brat'ev Kashirinyh st. 129\\Chelyabinsk  454021, Russia}
\address[02]{Krasovskii Institute of Mathematics and Mechanics, S. Kovalevskaja st. 16 \\ Yekaterinburg
620990, Russia}


\begin{abstract}
In this paper we survey existing results on Deza graphs and give some new results. We present an introduction to Deza  graphs  for  the  reader  who  is  unfamiliar  with  the  subject,  and  then  give an overview of some developments in the area of Deza graphs since the initial paper by five authors [M. Erickson, S. Fernando, W. H. Haemers, D. Hardy, J. Hemmeter, Deza graphs: \emph{A generalization of strongly regular graphs}, J. Comb. Designs. 7 (1999), 395--405.] was written. We then investigate 3-class cyclotomic schemes and give necessary and sufficient conditions to get a Deza graph as a graph given by one relation or the union of two relations. Finally, we prove that a strictly Deza circulant on $2p$ vertices, where $p$ is prime, is isomorphic to the lexicographical product of the Paley graph on $p$ vertices with an edge.
\end{abstract}

\begin{keyword}
Deza graph; strongly regular graph;
\vspace{\baselineskip}
\MSC[2010] 05C25\sep 05E10\sep 05E15
\end{keyword}
\end{frontmatter}
\tableofcontents
\section{Survey}
\subsection{Introduction}
Michel Deza was one of the founding editors of the
European Journal of Combinatorics, and had a very big influence on it in its early days. Deza graphs were introduced in 1999 in the initial paper \cite{EFHHH99}. The European Journal of Combinatorics celebrates its 40th birthday in 2020.

In this paper we survey existing results on Deza graphs and give some new results. We hope it would be an appropriate tribute. The paper is organised as follows. In Section 1 we survey existing results on Deza graphs.
In Section 2.1 we give a construction of strictly Deza graphs based on cyclotomic association schemes. In Section 2.2 we show that a strictly Deza circulant on $2p$ vertices, where $p$ is prime, is necessarily the lexicographical product of the Paley graph on $p$ vertices and an edge.

\subsection{Preliminaries}
Let $\Gamma = (V,E)$ be a graph.
For an edge $\{u, v\}$, the vertices $u$ and $v$ are said to be \emph{adjacent} to one another, which is denoted by $u \sim v$.
The length of a shortest path connecting vertices $u, v$ is called the \emph{distance} between
the vertices $u, v$ and denoted by $d(u,v)$.
The \emph{diameter} of a graph $\Gamma$ is the maximum distance between two vertices.

A subgraph $\Delta$ of a graph $\Gamma$ is called \emph{induced}, if for any two vertices
$u,v$ in $\Delta$, $u$ and $v$ are adjacent if and only if they are adjacent in $\Gamma$.

For any vertex $v$, define the \emph{neighbourhood} $N(v)$ and
the \emph{second neighbourhood} $N_2(v)$ as the subgraphs induced by the sets $\{u~|~u \sim v\}$ and $\{u~|~d(v,u) = 2\}$, respectively.
For any two vertices $u,v$
define the \emph{common neighbourhood} $N(u,v)$
of the vertices $u,v$ as the subgraph induced by the set $N(u) \cap N(v)$.

A graph is called \emph{regular of valency} $k$ if each its vertex has exactly $k$ neighbours.
A graph is called \emph{edge-regular} with parameters $(n,k,\lambda)$ if it has $n$ vertices, is regular of valency $k$ and
for any pair of adjacent vertices  $u,v$ the equality $|N(u,v)| = \lambda$ holds.
A graph is called \emph{coedge-regular} with parameters $(n,k,\mu)$ if it has $n$ vertices, is regular of valency $k$ and
for any pair of distinct non-adjacent vertices  $u,v$ the equality $|N(u,v)| = \mu$ holds.
A graph is called \emph{strongly regular} with parameters $(n,k,\lambda,\mu)$
if it is edge-regular with parameters $(n,k,\lambda)$ and coedge-regular with parameters $(n,k,\mu)$.
An edge-regular graph with parameters $(n,k,\lambda)$ is called \emph{quasi-strongly regular} with parameters $(n,k,\lambda;\mu_1,\mu_2)$ (see \cite{G06})
if for any pair of distinct non-adjacent vertices $u, v$, the number $|N(u,v)|$ takes precisely two values $\mu_1$ or $\mu_2$.

Let $n$, $k$, $b$, $a$ be integers such that the condition $0 \le a \le b \le k < n$ holds.
A graph $\Gamma$ is called a \emph{Deza graph} with parameters $(n,k,b,a)$ if it has $n$ vertices, is regular of valency $k$ and for any pair
of distinct vertices $u, v$ the number $|N(u,v)|$ takes precisely two values $a$ or $b$ (here we mean
that $\Gamma$ contains a pair of vertices having $a$ common neighbours and a pair of vertices
having $b$ common neighbours).
A Deza graph is a \emph{strictly Deza graph} if it has diameter $2$ and is not strongly regular (see \cite{EFHHH99}).
Note, that for a strictly Deza graph with parameters $(n, k, b, a)$ the inequality $a < b$ holds.

The class of Deza graphs thus generalises the class of strongly regular graphs. The classes of edge-regular strictly Deza graphs and coedge-regular strictly Deza graphs are important special cases of strictly Deza graphs.

Further, we give several more definitions related to Deza graphs.

A $k$-regular graph is a \emph{divisible design graph} if the vertex set can be
partitioned into $m$ classes of size $n$ such that two distinct vertices from the same class have exactly
$\lambda_1$ common neighbors, and two vertices from different classes have exactly $\lambda_2$ common neighbors.
Divisible design graphs form a special class of Deza graphs. This class of graphs was introduced in \cite{HKM11}. Walk-regular divisible design graphs were studied in \cite{CH14}. An infinite family of divisible design graphs was constructed in \cite{KS20}.

A connected loopless graph $\Delta$ is called a \emph{$(0,\lambda)$-graph} if any two distinct vertices in $\Delta$ have $\lambda$ common neighbours or none at all \cite{M79}. The class of Deza graphs with parameters $(n,k,b,0)$ coincide with the class of $(0,\lambda)$-graphs.
Deza graphs of diameter greater than $2$ have parameters $(n,k,b,0)$ and belong to the class of $(0,\lambda)$-graphs.

\subsection{Comparative analysis of properties}
Since Deza graphs were introduced as a generalisation of strongly regular graphs, in this section we consider several properties of strongly regular graphs and discuss if these properties hold for Deza graphs in general.

\subsubsection{Parameters $\alpha$ and $\beta$ for a Deza graph}
Let $\Gamma$ be a connected strongly regular graph with parameters $(n,k,\lambda,\mu)$ and $x$ be a vertex in $\Gamma$. Then the sets of vertices that have $\lambda$ common neighbours with $x$ (so called \emph{$\lambda$-vertices} for $x$) and $\mu$ common neighbours with $x$ (so called \emph{$\mu$-vertices} for $x$) are the first neighbourhood $N_1(x)$ and the second one $N_2(x)$, respectively. There is a generalisation of this for Deza graphs. Let $\Delta$ be a Deza graph with parameters $(n,k,b,a)$. For a vertex $x$, in a similar way, we may introduce so called \emph{$a$-vertices} and \emph{$b$-vertices} (as having $a$ and $b$ common neighbours with $x$, respectively).

\begin{proposition}[{[1, Proposition 1.1]}]
    For a Deza graph with parameters $(n,k,b,a)$, the numbers of $a$-vertices and $b$-vertices do not depend on the choice of the vertex $x$ and, for the case $b > a$, can be computed as follows:
    $$
    \alpha := \frac{b(n-1)-k(k-1)}{b-a},~~
    \beta := \frac{k(k-1) - a(n-1)}{b-a}.
    $$
\end{proposition}
For a vertex $x$, $N(x)$ and $N_2(x)$ may contain $a$-vertices as well as $b$-vertices.

\subsubsection{Complements of Deza graphs}
For a strongly regular graph, its complement is known to be strongly regular. But the same does not hold for Deza graphs in general.

Let $\Delta$ be a Deza graph with parameters $(n,k,b,a)$ that is not strongly regular. Then the following four situations are possible.

\begin{enumerate}
    \item There exist two non-adjacent vertices that have $a$ common neighbours.
    \item There exist two non-adjacent vertices that have $b$ common neighbours.
    \item There exist two adjacent vertices that have $a$ common neighbours.
    \item There exist two adjacent vertices that have $b$ common neighbours.
\end{enumerate}
\begin{proposition}
    Let $\Delta$ be a Deza graph  with parameters $(n,k,b,a)$ that is not strongly regular.
    The complement $\overline{\Delta}$ is a Deza graph if and only if $b = a + 2$ and either the situation 1 or the situation 4 does not hold.
\end{proposition}
\begin{proof}
Let $u,v$ be two vertices in $\Delta$.
Then, for the number $|N_{\overline{\Delta}}(u,v)|$ of common neighbours of $u$ and $v$ in $\overline{\Delta}$, we have
$$
|N_{\overline{\Delta}}(u,v)| =
\begin{cases}
 n-2k+a-2, & \text{if situation 1 holds;}\\
 n-2k+b-2, & \text{if situation 2 holds;}\\
 n-2k+a, & \text{if situation 3 holds;}\\
 n-2k+b, & \text{if situation 4 holds.}
\end{cases}
$$
We require $|\{n-2k+a-2, n-2k+b-2, n-2k+a, n-2k+b\}| = 2$.
Since $\Delta$ is not strongly regular, the inequality $a < b$ holds. It implies $n-2k+b-2 = n-2k+a$ and, thus, $b = a+2$. Moreover, either situation 1 or situation 4 must not hold. $\square$
\end{proof}

\subsubsection{Matrix conditions and Deza children}
One can define Deza graphs in terms of matrices. Suppose $\Delta$ is a graph with $n$
vertices, and $M$ is its adjacency matrix. Then $\Delta$ is a Deza graph with parameters $(n, k, b, a)$ if and only if
$$M^2 = aA + bB + kI$$
for some (0,1)-matrices $A$ and $B$ such that $A + B + I = J$, the all ones matrix. Note
that $\Delta$ is a strongly regular graph if and only if $A$ or $B$ is $M$.
The matrices $A$ and $B$ regarded as adjacency matrices define graphs called the \emph{Deza children} of $\Delta$ and denoted by $\Delta_a$ and $\Delta_b$.
It is easy to see that $\Delta$ is a strongly regular graph if and only if $\Delta$ is equal to $\Delta_a$ or $\Delta_b$.

\subsubsection{Vertex connectivity}
For a connected, not complete graph, the \emph{vertex connectivity number} (\emph{vertex connectivity}) is the minimum number of vertices whose removal from the graph makes it disconnected.

In \cite{BM85}, it was proved by Brouwer and Mesner that the  vertex connectivity of a connected strongly regular graph is equal to its valency.
In \cite{BK09}, Brouwer and Koolen showed that the vertex connectivity of a connected distance-regular graph is equal to its valency.

In \cite{BM97}, it was shown that the vertex connectivity of a $(0, 2)$-graph equals its valency.
In \cite{GGK14}, the vertex connectivity of strictly Deza graphs obtained by dual Seidel switching from a strongly regular graph was studied. It was shown that, if we the original strongly regular graph has eigenvalue $r \notin \{1,2\}$, then the vertex connectivity of the resulting strictly Deza graph is equal to its valency. In \cite{GP19}, it was shown that there exists infinitely many strictly Deza graphs with vertex connectivity $k-1$ where $k$ is the valency. They are obtained from strongly regular graphs with $r = 1$ by dual Seidel switching. In particular, it was shown that there exists infinitely many vertex-transitive strictly Deza graphs with vertex connectivity $k-1$ where $k$ is the valency. It gives motivation to continue studying the vertex connectivity of Deza graphs.

\subsubsection{Spectra of Deza graphs}
Any non-trivial strongly regular graph with parameters $(n,k,\lambda,\mu)$ has exactly three eigenvalues:
the principal eigenvalue $k$ and the non-principal eigenvalues $r$ and $s$, which can be expressed
in terms of the parameters $n,k,\lambda,\mu$, and the inequalities $s < 0 < r < k$ hold. It is also well-known that a regular graph with exactly three distinct eigenvalues is strongly regular (see \cite[Lemma 10.2.1]{GR01}).

Spectra of Deza graphs were firstly studied in \cite{AGHKKS20}.
In general, strictly Deza graphs can have
unbounded number of eigenvalues. In fact, for an integer  $d \geqslant 3$, let $H(d,2)$ be the $d$-dimensional hypercube, which is an edge-regular $(0,2)$-graph with $d+1$ distinct eigenvalues. Then the complementary graph $\overline{H(d,2)}$ is a strictly Deza graph with parameters $(2^d, 2^d - d -1, 2^d - 2d, 2^d-2d -2)$ having the $d+1$ distinct eigenvalues $2^d-d-1$, $-1-(d-2i)$, where $1 \leqslant i \leqslant d$. On the other hand, for a divisible design graph its spectrum can be computed in terms of the parameters of the graph. In particular, a divisible design graph has at most five distinct eigenvalues (see \cite[Lemma 2.1]{HKM11}).
Also it was remarkably shown that the spectra of Deza children of a Deza graph $\Delta$ with parameters $(n,k,b,a)$ can be expressed in terms of $n,k,b,a$ and eigenvalues of $\Delta$.
\begin{theorem}[{\cite[Theorem 3.2]{AGHKKS20}}]
Let $\Delta$ be a Deza graph with parameters $(n,k,b,a)$. Let $M,A,B$ be the adjacency matrices of $\Delta$ and its children, respectively. If $\theta_1 = k, \theta_2, \ldots, \theta_n$ are the eigevalues of $M$, then the eigenvalues of $A$ are
$$\alpha=\frac{b(n-1)-k(k-1)}{b-a}, \frac{k-b-\theta_2^2}{b-a},\ldots,\frac{k-b-\theta_n^2}{b-a}$$
and the eigenvalues of $B$ are
$$\beta=\frac{a(n-1)-k(k-1)}{a-b}, \frac{k-a-\theta_2^2}{a-b},\ldots,\frac{k-a-\theta_n^2}{a-b}.$$
\end{theorem}

Some further discussion on the nullity of Deza graphs and spectra of divisible design graphs can be found in \cite{AGHKKS20}.

\subsection{Constructions of Deza graphs}
\subsubsection{Dual Seidel switching}
Permuting of rows (but not columns) of the adjacency matrix of a graph is called \emph{dual Seidel switching} (see \cite{H84}). In \cite{EFHHH99}, dual Seidel switching was adapted to produce strictly Deza graphs from strongly regular graphs.

\begin{theorem}[{\cite[Theorem 3.1]{EFHHH99}}]\label{DezaDSS}
Let $\Gamma$ be a strongly regular graph with parameters $(n,k,\lambda,\mu)$, $k \ne \mu$, $\lambda \ne \mu$ and adjacency matrix $M$. Let $P$ be a permutation matrix. Then $PM$ is the adjacency matrix of a Deza graph $\Delta$ if and only if $P = I$ or $P$ represents an involution of $\Gamma$ that interchanges only non-adjacent vertices. Moreover, $\Delta$ is strictly Deza if $P \ne I$, $\lambda \ne 0$ and $\mu \ne 0$.
\end{theorem}

The resulting strictly Deza graph in Theorem \ref{DezaDSS} has the same parameters as the original strongly regular graph.
We also point out that if a strictly Deza graph has an order 2 automorhism that interchanges only non-adjacent vertices, then we can apply dual Seidel switching this strictly Deza graph to produce a Deza graph. Note that in this case the resulting graph may be strongly regular.

Another remark is that dual Seidel switching can be called an operation that does control the spectrum of the resulting graph. In fact, let $M$ and $PM$ be the adjacency matrices of a graph $\Gamma$ and the switched graph $\Delta$, where $P$ represents the permutation matrix of an order 2 automorphism of $\Gamma$. Since $P^{-1} = P$, we have $(PM)^2 = (PMP)M = M^2$, which means that the spectrum of $PM$ is the union of $\{\theta, -\theta\}$, where $\theta$ runs over the eigenvalues of $M$.

The following is a useful proposition, which shows what happens with neighbourhoods of vertices after dual Seidel switching.

\begin{proposition}\label{Images} Let $\Gamma$ be a graph and $\varphi$ be its order 2 automorhism interchanging only non-adjacent vertices. Let $\Delta$ be the graph obtained from $\Gamma$ by dual Seidel switching induced by $\varphi$.
For the neighbourhood $N_\Delta(x)$ of a vertex $x$ in the graph $\Delta$, the following conditions hold:
$$
N_\Delta(x) =
\left\{
  \begin{array}{ll}
    N_\Gamma(x), & \hbox{if $\varphi(x) = x$;} \\
    N_\Gamma(\varphi(x)), & \hbox{if $\varphi(x) \ne x$.}
  \end{array}
\right.
$$
\end{proposition}
\begin{proof}
    It follows from the definition of dual Seidel switching. $\square$
\end{proof}

\medskip

Given an integer $n \geqslant 3$, define the \emph{square lattice graph} $L(n)$ whose vertices are coordinates $(x,y)$, $1 \leqslant x \leqslant n, 1 \leqslant y \leqslant n$ with two distinct vertices $(x_1,y_1), (x_2,y_2)$ being adjacent whenever $x_1 = x_2$ or $y_1 = y_2$.
For any $i \in \{1, \ldots, n\}$, the sets $\{(i, j)~|~j \in \{1, \ldots, n\}\}$ and
$\{(j,i)~|~j \in \{1, \ldots, n\}\}$ are called the $i$th \emph{row} and the $i$th \emph{column} of $L(n)$, respectively.
Consider the mapping that sends any vertex $(x,y)$ to $(y,x)$ and call it by the \emph{reflection with respect to the main diagonal}. Also consider the mapping that sends any vertex $(x,y)$ to $(n+1 - x, n+1-y)$ and call it by the \emph{point reflection through the center}.

In \cite{KS10}, order 2 automorhisms of the square lattice graph $L(n)$ were studied. The following result was obtained.

\begin{theorem}[{\cite[Theorem 1]{KS10}}]\label{LatticeDSS}
If $n$ is even, then $L(n)$ has two non-equivalent order 2 automorphisms that interchange only non-adjacent vertices; they are represented by the reflection with respect to the main diagonal and the point reflection through the center. If $n$ is odd, then $L(n)$ has the only non-equivalent order 2 automorhism that interchanges only non-adjacent vertices; it is represented by the reflection with respect to the main diagonal.
\end{theorem}
Using Theorem \ref{LatticeDSS}, we apply dual Seidel switching to the lattice graphs, which gives strictly Deza graphs with parameters $(n^2,2(n-1),n-2,2)$; these parameters correspond to the parameters of the lattice graph as a strongly regular graph, which are $(n^2,2(n-1),n-2,2)$. We call the resulting graphs as the  \emph{quasi-lattice} strictly Deza graphs. It was determined in \cite[Lemma]{KS10} that the quasi-lattice strictly Deza graphs have two types of local subgraphs, which correspond to the fixed and to the moved vertices with respect to the order 2 automorphisms. It was then shown in \cite[Theorem 2]{KS10} that a strictly Deza graph obtained from $L(n)$ by dual Seidel switching, where the order 2 automorphism is the reflection with respect to the main diagonal, is characterised by its local subgraphs.

Given an integer $n \geqslant 5$, define the triangular graph $T(n)$ whose vertices are 2-element subsets in $\{1,\ldots,n\}$ with two vertices $U_1,U_2$ being adjacent whenever $|U_1 \cap U_2|=1$. For an even $n \geqslant 6$, consider the mapping that sends any vertex $\{x,y\}$ to $\{n+1-y,n+1-x\}$ and call it by the \emph{reflection with respect to the diagonal}. Note that this mapping fixes precisely the $\frac{n}{2}$ vertices $\{1,n\}, \{2,n-1\}, \ldots, \{\frac{n}{2},\frac{n}{2}+1\}$.

In \cite{S11}, order 2 automorhisms of the triangular graph $L(n)$ were studied. The following result was obtained.

\begin{theorem}[{\cite[Theorem 1]{S11}}]\label{TriangularDSS}
If $n$ is even, then $T(n)$ has the only non-equivalent order 2 automorphism that interchanges only non-adjacent vertices; it is represented by the reflection with respect to the diagonal. If $n$ is odd, then $T(n)$ has no order 2 automorphisms that interchange only non-adjacent vertices.
\end{theorem}

Using Theorem \ref{TriangularDSS}, we apply dual Seidel switching to the triangular graphs, which gives strictly Deza graphs with parameters $(\frac{n(n-1)}{2},2(n-2),n-2,4)$; these parameters correspond to the parameters of the triangular graph as a strongly regular graph, which are $(\frac{n(n-1)}{2},2(n-2),n-2,4)$. We call the resulting graphs as the \emph{quasi-triangular} strictly Deza graphs. It was determined in \cite[Lemma]{S11} that the quasi-triangular strictly Deza graphs have two types of local subgraphs, which correspond to the fixed and to the moved vertices with respect to the order 2 automorphisms. It was then shown in \cite[Theorem 2]{S11} that a strictly Deza graph that has the same local subgraphs as the quasi-triangular strictly Deza graph is a quasi-triangular strictly Deza graph.

In \cite{GS13}, order 2 automorhisms of the complements of the lattice graph $\overline{L(n)}$ and the triangular graph $\overline{T(n)}$ were studied.

For any $i \in \{1, \ldots, n\}$, the sets $\{(i, j)~|~j \in \{1, \ldots, n\}\}$ and
$\{(j,i)~|~j \in \{1, \ldots, n\}\}$ are called the $i$th \emph{row} and the $i$th \emph{column} of $\overline{L(n)}$, respectively. For any $i \in \{1,\ldots,\lfloor n/2\rfloor\}$, let us take the first $i$ pairs of rows in $\overline{L(n)}$
(the pairs of $1$st and $2$nd, $3$rd and $4$th, $\ldots$, $(2i-1)$th and $(2i)$th rows).
Then the permutation that swaps rows in each of the $i$ pairs is a an order 2 automorphism of $\overline{L(n)}$ that interchanges only non-adjacent vertices. Let us call such an automorphism of $\overline{L(n)}$ as the $i$-\emph{automorphism}.

\begin{theorem}[{\cite[Proposition 6]{GS13}}]\label{overLDSS}
The graph $\overline{L(n)}$ has $\lfloor \frac{n}{2}\rfloor$ non-equivalent order 2 automorphisms that interchange only non-adjacent vertices; they are represented by the $i$-automorphisms for $i \in \{1,\ldots,\lfloor \frac{n}{2}\rfloor\}.$
\end{theorem}

For any $i \in \{1,2,\ldots,n\}$, denote by $C(i)$ the maximal clique of $T(n)$ induced by
the set of all $2$-subsets that contain $i$. Note that for any distinct $i,j \in \{1,2,\ldots,n\}$
the equality $C(i) \cap C(j) = \{i,j\}$ holds.
The mapping that swaps the vertices $\{1,z\}$ to $\{2,z\}$, for all $z \in \{3,4,\ldots,n\}$,
is an order 2 automorphism of $\overline{T(n)}$ that interchanges only non-adjacent vertices. We call this automorphism of $\overline{T(n)}$ as the $\{1,2\}$-\emph{automorphism}.
\begin{theorem}[{\cite[Proposition 1]{GS13}}]\label{overTDSS}
The graph $\overline{T(n)}$ has
the only non-equivalent order 2 automorphism that interchanges only non-adjacent vertices; it is represented by the $\{1,2\}$-automorphism.
\end{theorem}

In \cite{Z19}, dual Seidel switching dual was applied to the Berlekamp-Van Lint-Seidel graph, which gave a new strictly Deza graph with parameters of the Berlekamp-Van Lint-Seidel graph.

We conclude this section with a remark that dual Seidel switching, together with strong product with an edge, was shown to be the only operation that can produce strictly Deza graph with parameters $(n,k,k-1,a)$ and $\beta = 1$. The characterisation requires existence of a strongly regular graph with parameters $\lambda - \mu = -1$ (see \cite{GHKS19}). Strictly Deza graphs obtained by dual Seidel switching from the Paley graphs of square order and from the Hoffman-Singleton graph were discussed in \cite{GHKS19}.

\subsubsection{Generalised dual Seidel switching}
In \cite{KKS21}, a general approach to dual Seidel switching was presented. Then this approach was applied to produce Deza graphs with strongly regular children.

\subsubsection{Deza graphs that are Cayley graphs}
Let $G$ be a group and $S$ be a generating set of $G$ such that $e_G \not\in S$ and $S^{-1} = S$ hold.
The graph with vertex set $G$, such that any vertices $x,y \in G$ are adjacent if and only if $xy^{-1} \in S$ holds, is called a
\emph{Cayley graph} of group $G$ \emph{with connection set} $S$ and is denoted by $Cay(G,S)$.

In \cite{EFHHH99}, the following criterion for a Cayley graph to be a Deza graph was proposed.
\begin{proposition}[{\cite[Proposition 2.1]{EFHHH99}}]
        A Cayley graph $Cay(G,S)$ is a Deza graph with parameters $(n,k,b,a)$ if and only if $|G| = n$, $|S| = k$ and $SS^{-1} = aA + bB + k\{e\}$ holds, where $A, B$ and $\{e\}$ is a partition of $G$.
\end{proposition}
In \cite{CR21}, a new infinite family of strictly Deza graphs that are Cayley graphs was constructed. The WL-rank of these graphs is equal to the number of vertices. The graphs from this family are divisible design graphs and integral.

In Section \ref{EnumRes}, we discuss the enumeration of strictly Deza graphs with $a > 0$ that are Cayley graphs and have from 8 to 59 vertices.
In Section \ref{NewRes}, we obtain some new results on strictly Deza graphs that are Cayley graphs over cyclic groups.

\subsubsection{Deza graphs from association schemes}
Let $X$ be a set of size $n$, and $R_0, R_1, \ldots, R_d$ be relations defined on $X$. Let $A_0, A_1, \ldots, A_d$ be the 0-1 matrices corresponding to the relations, that is, the $(x,y)$-entry of $A_i$ is 1 if and only if $(x,y) \in R_i$. Then $(X, \{R_i\}_{i = 0}^d)$ is called a \emph{$d$-class symmetric association scheme} if
    \begin{enumerate}
        \item $A_0 = I$;
        \item $\sum_iA_i = J$;
        \item each $A_i$ is symmetric;
        \item for each pair $i$ and $j$, $A_iA_j = \sum_k p_{ij}^kA_k$ for some constants $p_{ij}^k$.
    \end{enumerate}

    \begin{theorem}[{\cite[Theorem 4.2]{EFHHH99}}]\label{ASUnion}
    Let $(X, \{R_i\}_{i = 0}^d)$ be a symmetric association scheme, and $F \ subset \{1,2,\ldots, d\}$. Let $\Delta$ be the graph with adjacency matrix $\sum_{f \in F}A_f$. Then $\Delta$ is a Deza graph if and only if $\sum\limits_{f,g \in F}p_{f,g}^k$
    takes on at most two values, as $k$ ranges over $\{1,\ldots,d\}$.
    \end{theorem}

    In Section \ref{NewRes} we give a construction of strictly Deza graphs based on cyclotomic association schemes.

\subsubsection{Deza graphs as lexicographical products}
Let $\Gamma_1 = (V_1, E_1)$ and $\Gamma_2 = (V_2, E_2)$ be graphs. The \emph{lexicographical product} $\Gamma_1[\Gamma_2]$ of $\Gamma_1$
and $\Gamma_2$ is a graph with vertex set $V_1 \times V_2$, and adjacency defined by
$$(u_1,u_2) \sim (v_1,v_2) \text{~iff~} u_1 \sim u_2\text{~or~}(u_1 = v_1 \text{~and~} u_2 \sim v_2).$$

\begin{example}\label{ExCoedgeRegular}
Let $\Gamma_1 = K_x$, the complete graph on $x$ vertices and $\Gamma_2 = yK_2$, $y$ copies of an $K_2$. Then $\Gamma_1[\Gamma_2]$ is a coedge-regular strictly Deza graph with parameters $(2xy,1+2y(x-1),2y(x-1),2y(x-2)+2)$.
\end{example}

\begin{example}\label{ExEdgeRegular}
Let $\Gamma_1$ be a strongly regular graph with parameters $(n,k,\lambda,\lambda)$. Let $\Gamma_2 = \overline{K_{n'}}$, the coclique of size $n'$. Then $\Gamma_1[\Gamma_2]$ is an edge-regular strictly Deza graph with parameters $(nn',kn',kn',\lambda n')$.
\end{example}

\begin{example}
Let $\Gamma$ be a strongly regular graph with parameters $(n,k,\lambda,\mu)$, where $\lambda - \mu = -1$. Then the graph $\Gamma[K_2]$ is a strictly Deza graph with parameters $(2n,2k+1,2k,2\mu)$.
\end{example}

The following proposition gives a general condition for the lexicographical product of a strongly regular graph and a Deza graph to be a Deza graph.

\begin{proposition}[{\cite[Proposition 2.3]{EFHHH99}}]
        Let $\Gamma_1$ be a strongly regular graph with parameters $(n,k,\lambda,\mu)$ and $\Gamma_2$ be a Deza graph with parameters $(n',k',b,a)$. Then $\Gamma_1[\Gamma_2]$ is a $(k'+kn')$-regular graph on $nn'$ vertices. It is a Deza graph if and only if $$|\{a+kn',b+kn',\mu n',\lambda n'+2k'\}| \leqslant 2.$$
\end{proposition}

\subsubsection{Deza graphs based on symplectic and orthogonal graphs}
In \cite{LW08}, \cite{GWL10}, \cite{G13}, \cite{LWG13}, \cite{MP13} and \cite{A15}, several constructions of strictly Deza graphs based on symplectic graphs were proposed. Note that only \cite{LW08}, \cite{GWL10} and \cite{MP13} have constructions of strictly Deza graphs with $k \ne b$.

In \cite{GWZ13} and \cite{LGW14}, two constructions of strictly Deza graphs based on orthogonal graphs were proposed. Both constructions give strictly Deza graphs with $k = b$.

Note that a strictly Deza graph with $k = b$ is necessarily the lexicographical product of a strongly regular graph with $\lambda = \mu$ and a coclique of size $n' \geqslant 2$ (see \cite[Theorem 2.6]{EFHHH99}).

\subsubsection{Deza graphs as commuting graphs}
Let $V$ be the $n$-dimensional vector space over the finite field $\mathbb{F}_q$ for some positive integer $n$ and odd prime power $q$. Let $Y$ be a subgroup in $GL_n(\mathbb{F}_q)$.

It was shown in \cite[Section 3A]{H82} that, given an involution $w \in Y$, there are corresponding subspaces in $V$ defined as $V_w^+ = \{v \in V~|~w(v) = v\}$ and $V_w^- = \{v \in V~|~w(v) = -v\}$. In this setting, the decomposition $V = V_w^+ \oplus V_w^-$ holds. The \emph{type} of the involution $w$ is $dim V_w^-$.

Let $G$ be a group and $w$ be an involution in $G$. The \emph{commuting graph} $\Gamma_G(a)$ of the involution $w$ is defined on the set of involutions that are conjugate of $w$ in $G$, and two distinct vertices are adjacent whenever they commute.

In \cite{Z12}, a certain family of commuting graphs related to the unitary groups turned out to be strictly Deza graphs.
\begin{theorem}[\cite{Z12}]
Let $G = U_n(q)$, $q$ be an odd prime power, $n > 2$ and $w$ be an involution in $G$ of type 1. If $n$ is even, then the graph $\Gamma_G(w)$ is an edge-regular strictly Deza graph with parameters
$$(\frac{q^{n-1}(q^n-1)}{q+1}, \frac{q^{n-2}(q^{n-1}+1)}{q+1},
\frac{ q^{n-2}(q^{n-3}+1)}{(q+1)},   \frac{q^{n-3}(q^{n-2}-1)}{q+1}),$$
where $\lambda = \frac{q^{n-3}(q^{n-2}-1)}{q+1}$.
If $n$ is odd, then the graph $\Gamma_G(a)$ is an edge-regular strictly Deza graph with parameters
$$(\frac{q^{n-1}(q^n+1)}{q+1}, \frac{q^{n-2}(q^{n-1}-1)}{q+1},
\frac{q^{n-3}(q^{n-2}+1)}{q+1},
\frac{ q^{n-2}(q^{n-3}-1)}{q+1}),$$
where $\lambda = \frac{q^{n-3}(q^{n-2}+1)}{q+1}$.
\end{theorem}

\subsubsection{Deza graphs through $\pi$-local fusion graphs of finite simple groups of Lie-type of even characteristic}

A construction of Deza graphs through $\pi$-local fusion graphs of finite simple groups of Lie-type of even characteristic
was found in \cite{T20}.

\subsection{Deza graphs with restrictions}
Strongly regular graphs that have a vertex with disconnected second neighbourhood are known to be complete multipartite with parts of size at least 3 (see \cite[Lemma 3.1]{GGHR92}). In \cite{GIKMS17}, strictly Deza graphs that have a vertex with disconnected second neighbourhood were studied.

\begin{theorem}[{\cite[Theorem 1]{GIKMS17}}]
Let $\Delta$ be a strictly Deza graph. If the second neighborhood of each vertex is
disconnected, then $\Delta$ is either edge-regular or coedge-regular.
\end{theorem}

\begin{theorem}[{\cite[Theorem 2]{GIKMS17}}]\label{N2ERG}
Let $\Delta$ be an edge-regular strictly Deza graph that contains a vertex $x$ such that
the graph $N_2(x)$ is disconnected. Then $\Delta$ is the lexicographical product $\Gamma_1[\Gamma_2]$ of a strongly regular graph $\Gamma_1$ with
parameters $(n, k, \lambda, \mu)$, $\lambda = \mu$, and an $s$-coclique $\Gamma_2$ with $s \geqslant 2$.
\end{theorem}
Note that the graphs from Theorem \ref{N2ERG} coincide with the graphs from Example \ref{ExEdgeRegular}.

\begin{theorem}[{\cite[Theorem 3]{GIKMS17}}]\label{N2CERG}
Let $\Delta$ be a coedge-regular Deza graph of diameter 2 containing a vertex $x$ such
that the graph $N_2(x)$ is disconnected. Then $\Delta$ is the lexicographical product of a complete multipartite
graph with parts of the same size $s$ greater than 2 and an edge.
\end{theorem}
Note that the graphs from Theorem \ref{N2CERG} coincide with the graphs from Example \ref{ExCoedgeRegular}.

Deza graphs that do not contain $K_{1,3}$ as induced subgraphs were studied in \cite{KM14}, \cite{M14}, \cite{M17} and \cite{KM17}. In particular, strictly Deza graphs that are line graphs were listed in \cite{KM14}.

A connected graph $\Delta$ of even order is \emph{$\ell$-extendable} if it is of order at least $2\ell + 2$, contains
a matching of size $\ell$, and if every such matching is contained in a perfect matching of $\Delta$. In \cite{MS15}, 2-extendable Deza graphs of even order
and diameter 2 were classified. It turns out that there are only four Deza graphs of diameter 2 that are not 2-extendable.

Deza graphs of small Weisfeiler-Leman rank (WL-rank) were studied in \cite{BPR20}. In particular, Deza circulant graphs of WL-rank 4 were classified.

Spectra of Deza graphs with strongly regular children were studied in \cite{AHHKKS21}.

\subsection{Characterisations of Deza graphs with special parameters}
The following theorem gives a description of strictly Deza graphs with parameters $k = b$ in terms of strongly regular graphs with $\lambda = \mu$.

\begin{theorem}[{\cite{EFHHH99}[Theorem 2.6]}]
Let $\Delta$ be a strictly Deza graph with parameters $(n,k,b,a)$. The condition $b = k$ holds if and only if $\Delta = \Gamma_1[\Gamma_2]$ where $\Gamma_1$ is a strongly regular graph with parameters $(n_1,k_1,\lambda, \lambda)$ and $\Gamma_2$ is $\overline{K_{n_2}}$ for some $n_1,k_1,\lambda,n_2$. Moreover, the parameters satisfy
$$n = n_1n_2,~~k = b = k_1n_2,~~a = \lambda n_2,~~n_2 = \frac{k^2-an}{k-a}.$$
\end{theorem}

Let $\Delta$ be a strictly Deza graph with parameters $(n,k,k-1,a)$. In the case $\beta > 1$, $\Delta$ is shown to be the lexicographical product of a complete multipartite graph with parts of the same size and an edge (see \cite[Theorem]{KMS19}). In the case $\beta = 1$, $\Delta$ can be constructed from a strongly regular graph with parameters $\lambda - \mu = -1$ with use of certain two operations: the lexicographical product with an edge and, possibly, dual Seidel switching (see \cite[Theorem 2]{GHKS19}).

Deza graphs with parameters $(n,k,k-2,a)$ were listed in \cite{KS19}.

\subsection{Enumeration results}\label{EnumRes}
All strictly Deza graphs up to 13 vertices were listed in \cite{EFHHH99}. In \cite{GS11}, the list was extended up to 16 vertices. Then the list of strictly Deza graphs was extended \cite{GPS21} up to 21 vertices. The results are available by \url{http://alg.imm.uran.ru/dezagraphs/dezatab.html}.

All strictly Deza graphs with $a > 0$ that are Cayley graphs and have up to 59 vertices were listed in \cite{GS14}. The list contains 1272 graphs, which are available by \url{http://alg.imm.uran.ru/dezagraphs/deza_cayleytab.html}.

All (0,2)-graphs with valency at most 8 were listed in \cite{B06} and \cite{BO09}. The list is available by \url{https://www.win.tue.nl/~aeb/graphs/recta/02graphs.html}.

\subsection{Deza digraphs}
Deza digraphs were introduced and studied in \cite{ZW03}, \cite{WF06}, \cite{WF08} and \cite{CKS20}.

\subsection{Open problems}
In this section we list some open problems on Deza graphs.

\begin{problem}
What are strictly Deza graphs with parameters $(n,k,b,0)$?
\end{problem}

\begin{problem}
What are strictly Deza graphs that are not edge-regular, not coedge-regular and have a vertex with disconnected second neighbourhood?
\end{problem}

\begin{problem}[M. Muzychuk]
What are Deza graphs that admit an equitable 5-partition \{$x$, $A_1(x)$, $A_2(x)$, $B_1(x)$, $B_1(x)$\}, where $x$ is a vertex, $A(x)$ and $B(x)$ are $a$-vertices and $b$-vertices of $x$, $A_i(x) = A(x) \cap N_i(x)$, $B_i(x) = B(x) \cap N_i(x)$, $i \in \{1,2\}$? Are there such graphs that, for every vertex $x$, admit such an equitable 5-partition with the same quotient matrix?
\end{problem}

\begin{problem}
What is the vertex connectivity of Deza graphs? In particular, is there a strictly Deza graph with vertex connectivity less than $k-1$ where $k$ is the valency?
\end{problem}

\begin{problem}
What are Deza graphs that have four eigenvalues?
\end{problem}

\begin{problem}
For a strongly regular graph with parameters $(n,k,\lambda,\mu)$, $k \ne \mu$, $\lambda \ne \mu$, what are its order 2 automorphisms that interchange only non-adjacent vertices? Dual Seidel switching with respect to such an automorphism leads to a strictly Deza graph.
\end{problem}

\begin{problem}
What are strictly Deza graphs that are Cayley graphs? In particular, what are strictly Deza circulants?
\end{problem}

\section{New results on strictly Deza circulants}\label{NewRes}
\subsection{Introduction}
Strongly regular circulants were independently studied in \cite{BM79}, \cite{HLW79} and \cite{M84}. The following result was obtained.

\begin{theorem}[{\cite{BM79}, \cite{HLW79}, \cite{M84}}]\label{SRGCirculant}
Let $\Gamma$ be a strongly regular circulant. Then $\Gamma$ is isomorphic to a Paley graph $P(p)$ for some prime $p$, $p \equiv 1(4)$.
\end{theorem}

It was proved in \cite{MP03} that Paley graphs on a prime number of vertices are also the only non-trivial distance-regular circulants. Strongly regular Cayley graphs over $\mathbb{Z}_{p^n} \times \mathbb{Z}_{p^n}$ were studied in \cite{LM05}. Distance-regular Cayley graphs over dihedral and abelian groops were studied in \cite{MP07} and \cite{MS14}, respectively. In this paper we begin studying strictly Deza circulants.

\subsection{Preliminaries}
Let $\Gamma_1 = (V_1, E_1)$ and $\Gamma_2 = (V_2, E_2)$ be graphs.
The graph with the vertex set $V_1 \times V_2$, such that any two vertices $(u_1, u_2), (v_1,v_2)$ are adjacent
 if and only if either $u_1$ is adjacent to $v_1$ in $\Gamma_1$ or $u_1 = v_1$ and $u_2$ is adjacent to $v_2$ in $\Gamma_2$,
is called $\Gamma_2$-\emph{extension} of the graph $\Gamma_1$ and is denoted by $\Gamma_1[\Gamma_2]$. Note that the definition of $\Gamma_2$-\emph{extension} of the graph $\Gamma_1$ is equivalent to the lexicographical product of $\Gamma_1$ and $\Gamma_2$.
If $\Gamma_2$ is a $2$-clique $K_2$ then the graph $\Gamma_1[K_2]$ is called the $2$\emph{-clique-extension}
of the graph $\Gamma_1$.

Let $G$ be a group and $S$ be a generating set of $G$ such that $e_G \not\in S$ and $S^{-1} = S$ hold.
The graph with vertex set $G$, such that any vertices $x,y \in G$ are adjacent if and only if $xy^{-1} \in S$ holds, is called a
\emph{Cayley graph} of group $G$ \emph{with connection set} $S$ and is denoted by $Cay(G,S)$.
A graph that is isomorphic to a Cayley graph of a cyclic group is called a \emph{circulant}.

Let $q$ be a prime power, $q \equiv 1 (4)$. The Cayley graph of additive group $\mathbb{F}_q^+$ of the finite field $\mathbb{F}_q$,
whose generating set is the set of squares of the multiplicative group $\mathbb{F}_q^*$, is called the \emph{Paley graph} and
is denoted by $P(q)$. It is well-known that $P(q)$ is strongly regular with parameters $(q, \frac{q-1}{2}, \frac{q-5}{4}, \frac{q-1}{4})$.
Let $n$ be a positive integer, $n \equiv 1(4)$.
A strongly regular graph with parameters
$(n, \frac{n-1}{2}, \frac{n-5}{4}, \frac{n-1}{4})$ is called a \emph{conference graph}.

Let $V$ be a set of $n$ elements, ${\cal R}$ be a partition of $V\times V$ into $e+1$ binary
relations $R_0,R_1,\ldots,R_e$, which satisfy the following conditions:
\begin{itemize}
\item $R_0=\{(x,x)\mid x\in V\}$, the identity relation,
\item for any $i\in \{0,\ldots,e\}$, $R_i^{\top}=\{(y,x)\mid (x,y)\in R_i\}$ is a member of ${\cal R}$,
\item if $(x,y)\in R_k$, then the number of $z$ such that
$(x,z)\in R_i$ and $(z,y)\in R_j$
is a constant denoted by $p^k_{ij}$.
\end{itemize}
Then the pair $(V, {\cal R})$ is called an \emph{association scheme} of class $e$.
The numbers $p^k_{ij}$ are called the \emph{intersection numbers}.
An association scheme $(V,{\cal R})$ is called \emph{symmetric} if ${\cal R}$ consists of symmetric relations.
For an association scheme $(V,{\cal R})$ of class $e$, the numbers $n_0:=1, n_1 := p^1_{1,1}, \ldots, n_e := p^e_{e,e}$ are called the \emph{valencies}
of the scheme $(V,{\cal R})$. An association scheme $(V,{\cal R})$ of class $e$ is called \emph{pseudocyclic}, if
$n_1 = n_2 =   \ldots = n_e$ holds.

Let $q$ be a prime power, and $e$ be a divisor of $q-1$.
Fix a primitive element $\alpha$ of the the finite field
$\mathbb{F}_q$. Note that $\langle\alpha^e\rangle$ is a subgroup of $\mathbb{F}_q^*$ of index $e$
and its cosets are $\alpha^i\langle \alpha^e\rangle$, $(0\leq i\leq e-1)$.
Let us define the relations $R_0:=\{(x,x)\mid x\in \mathbb{F}_q\}$ and
$R_i:=\{(x,y)\mid x-y\in \alpha^i\langle \alpha^e\rangle, x,y\in \mathbb{F}_q\}~~(1\leq i\leq e).$
Then $(V,{\cal R}):=(\mathbb{F}_q, \{R_i\}_{i=0}^e)$ forms a pseudocyclic association scheme;
it is called the \emph{cyclotomic scheme} of class $e$ on $\mathbb{F}_q$.

\begin{lemma}\label{SymCyclScheme}
The cyclotomic scheme of class $e$ on $\mathbb{F}_q$ is symmetric if and only
if $q$ or $(q-1)/e$ is even.
\end{lemma}
\begin{proof}
It follows from the definition of the cyclotomic scheme. $\square$
\end{proof}

\begin{lemma}[{\cite[Lemma 10.3.4]{GR01}}]\label{SRGpvertices}
Let $\Gamma$ be a strongly regular graph with $p$ vertices, where $p$ is prime. Then $\Gamma$ is a conference graph.
\end{lemma}

\begin{lemma}[\cite{H84a}]\label{PseudoCyclic}
Let $(V,{\cal R})$ be a cyclotomic scheme of class $3$ on $\mathbb{F}_q$, where $q = 3t+1$.
Then the following statements hold.

{\rm (1)} The intersection numbers can be expressed in numbers $s,r,t$ as follows:\\
\begin{tabular}{cccc}
  \hline
  $p_{ij}^1$ & {\rm 1} & {\rm 2} & {\rm 3} \\
  \hline
  {\rm 1} & $t-r-s-1$ & $s$ & $r$ \\
  {\rm 2} & $s$ & $r$ & $t-r-s$ \\
  {\rm 3} & $r$ & $t-r-s$ & $s$ \\
  \hline
\end{tabular},\\
\begin{tabular}{cccc}
  \hline
  $p_{ij}^2$ & {\rm 1} & {\rm 2} & {\rm 3} \\
  \hline
  {\rm 1} & $s$ & $r$ & $t-r-s$ \\
  {\rm 2} & $r$ & $t-r-s-1$ & $s$ \\
  {\rm 3} & $t-r-s$ & $s$ & $r$ \\
  \hline
\end{tabular}, \\
\begin{tabular}{cccc}
  \hline
  $p_{ij}^3$ & {\rm 1} & {\rm 2} & {\rm 3} \\
  \hline
  {\rm 1} & $r$ & $t-r-s$ & $s$ \\
  {\rm 2} & $t-r-s$ & $s$ & $r$ \\
  {\rm 3} & $s$ & $r$ & $t-r-s-1$ \\
  \hline
\end{tabular},\\~\\ where $r,s$ and $t$ satisfy the additional condition
\begin{equation}\label{add}
1+2(r+s)-3(r-s)^2 = (1+3(r+s) - 2t)^2.
\end{equation}

{\rm (2)} The intersection numbers are invariant under a cyclic shift of the relations: $p_{ij}^k=p_{i+1,j+1}^{k+1}$
(the indices are taken mod 3)
\end{lemma}

\begin{lemma}[{\cite[Lemma 3.5]{H84a}}]\label{Reformulation}
The additional condition (\ref{add}) from Lemma \ref{PseudoCyclic} is equivalent to the equation

\begin{equation}\label{addEq}
L^2 + 27M^2 = 4q,
\end{equation}

where $L = 6t-2-9(r+s), M = r - s$  and $q = 3t+1$ hold. The equation (\ref{addEq}) always has a unique solution in integers $L, M$ (apart from the signs).

\end{lemma}

\begin{lemma}[\cite{C93}]\label{DiophantineEquation}
Let $m$ be an integer, $m \ge 3$. Then the following statements hold.

{\rm (1)} The diophantine equation $x^2+3 = y^m$ has no solutions;

{\rm (2)} The only solution of the diophantine equation $x^2+12 = y^m$ is $x = 2, y = 2, m = 4$.
\end{lemma}

\subsection{Strictly Deza graphs from 3-class cyclotomic schemes}

In this section we prove the following theorem, which gives a construction of strictly Deza graphs from 3-class cyclotomic association schemes. This construction is based on Theorem \ref{ASUnion}.

\begin{theorem}\label{FromCyclotomicScheme}
Let $(V, {\cal R}) = (\mathbb{F}_q, \{R_0, R_1, R_2, R_3\})$ be the symmetric cyclotomic scheme of class $3$ over the finite field $\mathbb{F}_q$,
where $q$ is a prime power. For any $i \in \{1,2,3\}$,  denote by $\Gamma_i$ and $\overline{\Gamma_i}$ the graphs on the vertex set $V$
such that the adjacency relation of $\Gamma_i$ is $R_i$ and the adjacency relation of $\overline{\Gamma_i}$ is the union of the relations
$R_j$ and $R_k$, where $\{i, j, k\} = \{1,2,3\}$.
Then the following statements hold.

{\rm (1)} For any $i \in\{1,2,3\}$, $\Gamma_i = Cay(\mathbb{F}_{q}^+, \alpha^iS)$ and $\overline{\Gamma_i} =
Cay(\mathbb{F}_{q}^+, \mathbb{F}_q^*\setminus \alpha^iS)$ hold,
where $\alpha$ is a primitive element of $\mathbb{F}_q$ and $S = \langle\alpha^3\rangle$ is the subgroup of index $3$ in $\mathbb{F}_q^*$.

{\rm (2)} The graphs $\Gamma_1, \Gamma_2$ and $\Gamma_3$ are isomorphic;

{\rm (3)} The graphs $\overline{\Gamma_1}, \overline{\Gamma_2}$ and $\overline{\Gamma_3}$ are isomorphic;

{\rm (4)} The graph $\Gamma_1$ is a strictly Deza graph if and only if $q$ is prime and there exists an integer $x$ such that $q = x^2+3$;

{\rm (5)} The graph $\overline{\Gamma_3}$ is a strictly Deza graph if and only if $q$ is prime and
there exists an integer $x$ such that $q = x^2+12$.
\end{theorem}
\begin{proof}
(1) Let us show that $S = -S$ holds, where $S$ is the subgroup of index $3$ in $\mathbb{F}_q^*$. It is enough to prove that
$S$ contains $-1$. If $q$ is even, then $-1 = 1 \in S$. Assume that $q$ is odd.
Since $3$ divides $(q-1)$, $3$ divides $(q-1)/2$.
So, we have $-1  = \alpha^\frac{q-1}{2} \in S$. This implies that $-\alpha S = \alpha S$ and $-\alpha^2 S = \alpha^2 S$.
Now, the equalities $\Gamma_i = Cay(\mathbb{F}_q^+, \alpha^i S)$, $i \in \{1,2,3\}$ follow from the definitions of a Cayley graph and a cyclotomic scheme.

(2) It can be proved by definition that the mapping $x \rightarrow \alpha^{j-i}x$ is an isomorphism between the graphs
$\Gamma_i$ and $\Gamma_j$, where $i,j \in \{1,2,3\}$.

(3) It follows from item (2) and the fact that $\overline{\Gamma_i}$ is the complement of $\Gamma_i$.

(4)
Suppose that the graph $\Gamma_1$ is a Deza graph.
Then the set of intersection numbers $\{p_{11}^1, p_{11}^2, p_{11}^3\} = \{t-r-s-1, s, r\}$ of the graph $\Gamma_1$ has the cardinality at most $2$.
So, there are three possible cases.

If $s = p_{11}^2 = p_{11}^3 = r$ holds, then the graph $\Gamma_1$ would be strongly regular, which is a contradiction with Lemma \ref{SRGpvertices}.

If $t-s-r-1 = p_{11}^1 = p_{11}^3 = r$ holds, then we have $s = t - 2r - 1$.

Let us prove that $s \ne r$. Suppose to the contrary that $s = r$ holds. It follows from the equality $s = t - 2r - 1$
and Lemma \ref{Reformulation} that $t-3r = 1$, $M  = 0$ and $L = 6t-2-18r = 3(t-3r) - 2 = 1$.
Thus, $4q = L^2 + 27M^2 = 1$. A contradiction.

By Lemma \ref{Reformulation}, we obtain $M = 3r - t + 1$ and $L = 6t-2-9(r+t-2r-1)= -3t + 7 + 9r = 3(3r-t+1) + 4 = 3M+4$.
Now let us take into account the condition (\ref{addEq}) from Lemma \ref{Reformulation}. We obtain
$$4q = (3M+4)^2+27M^2 = 36M^2 + 24M + 16,$$
$$q = 9M^2 + 6M + 4 = (3M+1)^2 + 3.$$

The case $t-s-r-1 = p_{11}^1 = p_{11}^2 = s$ is analogous to the previous one.

Thus, if $\Gamma_1$ is a strictly Deza graph, then we have $p_{11}^2 \ne p_{11}^3$ and exactly one of the equalities
$p_{11}^1 = p_{11}^2$ and $p_{11}^1 = p_{11}^3$ holds. Moreover, there exists $x := (3M + 1)$ such that $q = x^2+3$,
where $M = r - s$.
It follows from Lemma \ref{DiophantineEquation}(1) that, if $p^h = q = x^2 +3$ holds for some integer $x$,
then $h \le 2$. If $h = 2$, then $x = 1, y = 2$ is the unique solution of the equation $y^h = x^2+3$.
So, we can assume that $q$ is prime.

Now, let us prove that, if $q$ is prime and there exists an integer $x$ such that $q = x^2 + 3$ holds, then $\Gamma_1$ is a strictly Deza graph. Let $q = p^h$ hold, where $p$ is prime.
Note that we can choose $x$ such that $x \equiv 1(3)$ holds since $x$ is not divided by 3 and $x$, $-x$ are not equivalent modulo 3.
Put $M_1 := \frac{x-1}{3}$. Note that $M_1 \ne 0$. So, we have $x = 3M_1 + 1$ and, consequently,
$$(3M_1+1)^2 = x^2 = q - 3,$$
$$9M_1^2 + 6M_1 + 4 = q,$$
$$36M_1^2 + 24M_1 + 16 = 4q,$$
$$9M_1^2 + 24M_1 + 16 + 27M_1^2= 4q,$$
$$(3M_1 + 4)^2 + 27M_1^2= 4q.$$
We obtain that the numbers $L_1:=3M_1+4$ and $M_1$ form a solution of the equation (\ref{addEq}); by Lemma \ref{Reformulation}, this solution,
apart from the signs of $M_1$ and $L_1$,
is uniquely determined.
So we have the following four possible cases:
\begin{enumerate}
    \item $M_1 = M = r - s$, $L_1 = L =  6t-2-9(r+s)$ and $3(r-s)+4 = 6t-2-9(r+s)$;

    \item $M_1 = - M = -r + s$, $L_1 = L =  6t-2-9(r+s)$ and $3(-r+s)+4 = 6t-2-9(r+s)$;

    \item $M_1 = M = r - s$, $L_1 = -L =  -6t+2+9(r+s)$ and $3(r-s)+4 = -6t+2+9(r+s)$;

    \item $M_1 = - M = -r + s$, $L_1 = -L =  -6t+2+9(r+s)$ and $3(-r+s)+4 = -6t+2+9(r+s)$.
\end{enumerate}
Let us consider these cases.

1. We obtain $t-2r-s-1 = 0$ and $p_{11}^3 = r = t-r-s-1 = p_{11}^1$. Since $p_{11}^3-p_{11}^2 = r - s = M = M_1 \ne 0$, we conclude that $\Gamma_1$ is a strictly Deza graph.

2. We obtain $t-r-2s-1 = 0$ and $p_{11}^2 = s = t-r-s-1 = p_{11}^1$. Since $p_{11}^3-p_{11}^2 = r - s = M = -M_1 \ne 0$, we conclude that $\Gamma_1$ is a strictly Deza graph.

3. We obtain $6s - 3r - 3t = 1$, which is a contradiction since 3 divides the left part and does not divide the right one.

4. We obtain $6r - 3s - 3t = 1$, which is a contradiction since 3 divides the left part and does not divide the right one.

(5) Suppose that the graph $\overline{\Gamma_3}$ is a Deza graph.
Then the set of intersection numbers
$\{p_{11}^1+p_{12}^1+p_{21}^1+p_{22}^1, p_{11}^2+p_{12}^2+p_{21}^2+p_{22}^2, p_{11}^3+p_{12}^3+p_{21}^3+p_{22}^3\}
= \{t+s-1, t+r-1, 2t-r-s\}$ of the graph $\overline{\Gamma_3}$
has the cardinality at most $2$.
So, there are three possible cases.

The equality $t+s-1 = p_{11}^1+p_{12}^1+p_{21}^1+p_{22}^1 = p_{11}^2+p_{12}^2+p_{21}^2+p_{22}^2 = t+r-1$ is equivalent to $r = s$;
in this case the graph $\overline{\Gamma_3}$ would be strongly regular, which is a contradiction with Lemma \ref{SRGpvertices}.

If $t+r-1 = p_{11}^2+p_{12}^2+p_{21}^2+p_{22}^2 = p_{11}^3+p_{12}^3+p_{21}^3+p_{22}^3 = 2t-r-s$ holds, then we have $s = t - 2r+1$.
Note that $s \ne r$.
By Lemma \ref{Reformulation}, we obtain $M = 3r - t - 1$ and $L = 6t-2-9(r+t-2r+1)= -3t -11 + 9r = 3(3r-t-1) - 8 = 3M-8$.
Now let us take into account the condition (\ref{addEq}) from Lemma \ref{Reformulation}. We obtain
$$4q = (3M-8)^2+27M^2 = 36M^2 -48M + 64,$$
$$q = 9M^2 - 12M + 16 = (3M-2)^2 + 12.$$

The case $t+s-1 = p_{11}^1+p_{12}^1+p_{21}^1+p_{22}^1 = p_{11}^3+p_{12}^3+p_{21}^3+p_{22}^3 = 2t-r-s$ is analogous to the previous one.

Thus, if $\overline{\Gamma_3}$ is a strictly Deza graph, then we have
$p_{11}^1+p_{12}^1+p_{21}^1+p_{22}^1 \ne p_{11}^2+p_{12}^2+p_{21}^2+p_{22}^2$ and exactly one of the equalities
$p_{11}^1+p_{12}^1+p_{21}^1+p_{22}^1 = p_{11}^3+p_{12}^3+p_{21}^3+p_{22}^3$ and
$p_{11}^2+p_{12}^2+p_{21}^2+p_{22}^2 = p_{11}^3+p_{12}^3+p_{21}^3+p_{22}^3$ holds. Moreover, there exists $x := (3M - 2)$ such that $q = x^2+12$,
where $M = r-s$. It follows from Lemma \ref{DiophantineEquation}(2) that, if $p^h = q = x^2 + 12$ holds for some integer $x$,
then either $p = 2, n = 4, x = \pm2$; or $h \le 2$. If $h = 2$, then  $x = \pm2, y = 4$ is the unique solution of the equation $y^n = x^2+12$;
it can be shown that in this case $\overline{\Gamma_3}$ is the Clebsh graph, which is strongly regular.
So, we can assume that, if $q = x^2 + 12$ holds, then $q$ is prime.

Now, let us prove that, if $q$ is prime and there exists an integer $x$ such that $q = x^2 + 12$ holds, then $\overline{\Gamma_3}$ is a strictly Deza graph.
Let $q = p^h$ hold, where $p$ is prime.
Note that we can choose $x$ such that $x \equiv 1(3)$ holds since $x$ is not divided by 3 and $x$, $-x$ are not equivalent modulo 3.
Put $M_2 := \frac{x+2}{3}$. Note that $M_2 \ne 0$. So, we have $x = 3M_2 -2$ and, consequently,
$$(3M_2-2)^2 = x^2 = q - 12,$$
$$9M_2^2 - 12M_2 + 16 = q,$$
$$36M_2^2 - 48M_2 + 64 = 4q,$$
$$9M_2^2 - 48M_2 + 64 + 27M_2^2= 4q,$$
$$(3M_2 - 8)^2 + 27M_2^2= 4q.$$
We obtain that the numbers $L_2:=3M_2-8$ and $M_2$ form a solution of the equation (\ref{addEq}); by Lemma \ref{Reformulation}, this solution,
apart from the signs of $L_2$ and $M_2$,
is uniquely determined.
So we have the following four possible cases:
\begin{enumerate}
    \item $M_2 = M = r - s$, $L_1 = L =  6t-2-9(r+s)$ and $3(r-s)-8 = 6t-2-9(r+s)$;

    \item $M_2 = - M = -r + s$, $L_1 = L =  6t-2-9(r+s)$ and $3(-r+s)-8 = 6t-2-9(r+s)$;

    \item $M_2 = M = r - s$, $L_1 = -L =  -6t+2+9(r+s)$ and $3(r-s)-8 = -6t+2+9(r+s)$;

    \item $M_2 = - M = -r + s$, $L_1 = -L =  -6t+2+9(r+s)$ and $3(-r+s)-8 = -6t+2+9(r+s)$.
\end{enumerate}
Let us consider these cases.

1. We obtain $t - 2r -s+1 = 0$ and
$p_{11}^2+p_{12}^2+p_{21}^2+p_{22}^2 = t+r-1= 2t - r-s =
p_{11}^3+p_{12}^3+p_{21}^3+p_{22}^3$. Since $(p_{11}^2+p_{12}^2+p_{21}^2+p_{22}^2)-(p_{11}^1+p_{12}^1+p_{21}^1+p_{22}^1) = r - s = M = M_2 \ne 0$,
we conclude that $\overline{\Gamma_3}$ is a strictly Deza graph.

2. We obtain $t - r -2s+1 = 0$ and
$p_{11}^1+p_{12}^1+p_{21}^1+p_{22}^1 = t+s-1= 2t - r-s =
p_{11}^3+p_{12}^3+p_{21}^3+p_{22}^3$.
Since $(p_{11}^2+p_{12}^2+p_{21}^2+p_{22}^2)-(p_{11}^1+p_{12}^1+p_{21}^1+p_{22}^1) = r - s = M = -M_2 \ne 0$,
we conclude that $\overline{\Gamma_3}$ is a strictly Deza graph.

3. We obtain $6 +3r+6s - 3t=1$, which is a contradiction since 3 divides the left part and does not divide the right one.

4. We obtain $6 +6r+3s - 3t=1$,
which is a contradiction since 3 divides the left part and does not divide the right one.
$\square$

\end{proof}

\subsection{Strictly Deza circulants on $2p$ vertices}
In this section we study strictly Deza circulants on $2p$ vertices and show that such a graph is necessarily the 2-clique-extension $P(p)[K_2]$ of the Paley graph $P(p)$.

For any $t_1, t_2 \in \mathbb{Z}_{2p}$, denote by $\psi_{t_1, t_2}$ the mapping
that sends an element $x\in \mathbb{Z}_{2p}$ to $t_1x+t_2 \in \mathbb{Z}_{2p}$.

\begin{lemma}\label{AutLemma}
For a circulant $\Gamma = Cay(\mathbb{Z}_{2p}, S)$, the following conditions hold

{\rm (1)} $\psi_{-1,0} \in Aut(\Gamma)$;

{\rm (2)} $\{\psi_{1, t_2}~|~t_2 \in \mathbb{Z}_{2p}\} \le Aut(\Gamma)$;

{\rm (3)} Let $\psi$ be an automorphism of the graph $\Gamma$. Then, for any vertices $x,y \in \Gamma$,
the equality $|N(x,y)| = |N(\psi(x), \psi(y))|$ holds.
\end{lemma}
\begin{proof}
(1) It follows from the fact that $-S = S$.

(2) It follows from the definition of a circulant.

(3) It follows from the definition of an automorphism. $\square$
\end{proof}

\begin{proposition}\label{PaleyK2}
Let $\Gamma$ be the $2$-clique-extension of a conference graph $\Gamma_1$ with parameters $(n, \frac{n-1}{2}, \frac{n-5}{4}, \frac{n-1}{4})$.
The following statements hold.

{\rm (1)} $\Gamma$ is a strictly Deza graph with parameters $(2n, n, n-1, \frac{n-1}{2})$;

{\rm (2)} If $\Gamma_1$ is a Paley graph of order $q = p^m$,
where $p$ is prime, $m \ge 1$ and $q \equiv 1~(mod~4)$,
then $\Gamma_1[K_2]$ is isomorphic to $Cay(\mathbb{Z}_{2} \times \mathbb{F}_q^+, (\mathbb{Z}_{2} \times S_q) \cup \{(1,0)\})$,
where $S_q = \{x^2~|~x \in \mathbb{F}_{q}^*, x \ne 0\}$.
\end{proposition}
\begin{proof}

(1) There are two types of edges in $\Gamma$.
Let $\{x_1,x_2\}$ be an edge of the graph $\Gamma$ that extends a vertex $x'$ of $\Gamma_1$.
We consider $\{x_1,x_2\}$ as an arbitrary edge of the first type.
Note that $|N_\Gamma(x_1, x_2)| = |N_{\Gamma_1}(x')[K2]| = 2\frac{n-1}{2} = n-1$ holds.

Let $\{x_1,x_2\}$ and $\{y_1,y_2\}$ be edges of the graph $\Gamma$ that extend adjacent vertices $x'$ and $y'$ of $\Gamma_1$, respectively.
Note that by definition of the extension the pairs $\{x_1, y_1\}, \{x_1, y_2\}, \{x_2, y_1\}, \{x_2, y_2\}$
are edges in $\Gamma$.
We consider $\{x_1, y_1\}$ as an arbitrary edge of the second type.
Since $x',y'$ are adjacent in $\Gamma_1$, the equality $|N_{\Gamma_1}(x',y')| = \frac{n-5}{4}$ holds.
Then we have $|N_\Gamma(x_1, y_1)| = |N_{\Gamma_1}(x',y')[K2] \cup \{x_2, y_2\}| = 2\frac{n-5}{4} + 2 = \frac{n-1}{2}$.

Let $\{x_1,x_2\}$ and $\{y_1,y_2\}$ be edges of the graph $\Gamma$ that extend non-adjacent vertices $x'$ and $y'$ of $\Gamma_1$.
We consider $\{x_1, y_1\}$ as an arbitrary pair of non-adjacent vertices in $\Gamma$.
Since $x',y'$ are not adjacent in $\Gamma_1$, the equality $|N_{\Gamma_1}(x',y')| = \frac{n-1}{4}$ holds.
Then we have $|N_\Gamma(x_1, y_1)| = |N_{\Gamma_1}(x',y')[K2]| = 2\frac{n-1}{4} = \frac{n-1}{2}$.
Thus, $\Gamma$ is a coedge-regular strictly Deza graph.

(2) Let us consider the graph $Cay(\mathbb{Z}_{2} \times \mathbb{F}_q^+, (\mathbb{Z}_{2} \times S_q) \cup \{(1,0)\})$,
where $S_q = \{x^2~|~x \in \mathbb{F}_{q}^*, x \ne 0\}$. Note that the sets of vertices $\{0\}\times \mathbb{F}_q^+$
and $\{1\}\times \mathbb{F}_q^+$ induce two copies of the Paley graph of order $q$.
The element $(1,0) \in S$ connects the corresponding vertices of the two copies. Finally, by the definition of the set $S$,
if arbitrary vertices $(i, x), (i, y)$ are adjacent, then there are all possible edges between the sets
$\{(0, x), (1, x)\}$ and $\{(0, y), (1, y)\}$. Thus, the graph $Cay(\mathbb{Z}_{2} \times \mathbb{F}_q^+, (\mathbb{Z}_{2} \times S_q) \cup \{(1,0)\})$
is isomorphic to the 2-clique-extension of the Paley graph of order $q$. $\square$
\end{proof}

\begin{corollary}
The $2$-clique-extension $P(p)[K_2]$ of the Paley graph $P(p)$ is a strictly Deza circulant with parameters $(2p, p, p-1, \frac{p-1}{2})$.
\end{corollary}
\begin{proof}
For a prime $p$, the additive group $\mathbb{F}_p^+$ is a cyclic group of order $p$. The direct product of two cyclic groups of coprime order is a cyclic group. $\square$
\end{proof}

\medskip
The main result of this section is the following theorem.

\begin{theorem}\label{circ2p}
Let $\Gamma$ be a strictly Deza circulant of with parameters $(2p,k,b,a)$, where $p$ is prime. Then $p \equiv 1(4)$ holds, and $\Gamma$ is isomorphic
to the $2$-clique-extension $P(p)[K_2]$ of the Paley graph $P(p)$. In particular, $k = p,~b = p-1$ and $a = \frac{p-1}{2}$.
\end{theorem}
\begin{proof}
Consider a strictly Deza circulant $\Gamma = Cay(\mathbb{Z}_{2p}, S)$. Note that for elements of $\mathbb{Z}_{2p}$
the parity is well-defined. On our way, we prove several inner lemmas and then complete the proof of Theorem \ref{circ2p}.

\begin{lemma}\label{ElemsOfDifftParityHaveEvenNumOfComNeigh}
The following statements hold.

{\rm (1)} Let $x \in \mathbb{Z}_{2p}$ be an odd element. Then $|N_\Gamma(0,x)|$ is even.

{\rm (2)} Let $x\in \mathbb{Z}_{2p}$ be an odd element and $y\in \mathbb{Z}_{2p}$ be an even element. Then $|N(x,y)|$ is even.
\end{lemma}
\proof
(1) Let $s_1\in\mathbb{Z}_{2p}$ be a common neighbour of $0$ and $x$. So, the elements $s_1$ and $s_2:=x-s_1$ belong to $S = N(0)$.
Note, that $s_1 \ne s_2$ holds because $s_1$ and $s_2$ have different parity.
Since $x = s_1 + s_2$ holds, we have $x-s_2 = s_1 \in S$. Thus, $s_2$ belongs to $N(x)$, and we conclude that
$s_2$ is a common neighbour of the vertices $0$ and $x$. We obtain that each common neighbour $s_1$ of the vertices $0$ and $x$
gives another common neighbour $s_2$.
Note that the obtained pair of common neighbours $\{s_1, s_2\}$
is uniquely determined by each of the elements $s_1$ and $s_2$.
So, the set of common neighbours of the vertices $0$ and $x$ can be divided into disjoint pairs, and (1) is proved.

(2)
It follows from (1), Lemma \ref{AutLemma}(3) and from the fact that the automorphism $\psi_{1,-x} \in Aut(\Gamma)$ sends the pair
$(x,y)$ to the pair $(0,x-y)$, where $x-y$ is even. $\square$

\medskip

Consider the sets $S_4:= \{s \in S~|~s+p \in S\}$ and $S_2 := \{s \in S~|~ s \ne p,~s+p \not\in S\}$.
Note that $S$ is a disjoint union of $S_4$, $S_2$ and, maybe, the set $\{p\}$ (if $p \in S$ holds).

\begin{lemma}\label{S4IsSetOfComNei0andp}
The following statements hold.

{\rm (1)} $|S_4|$ is divided by $4$.

{\rm (2)} $S_4$ is the set of common neighbours of the vertices $0$ and $p$.
\end{lemma}
\proof
(1) Let us consider an element $s \in S_4$. By the definition of $S_4$, we have $s + p \in S_4$.
By the definition of a Cayley graph, we have $-s \in S$ and $-(s+p) = p - s \in S$. This means that
$-s$ and $p-s$ belong to $S_4$. Note that the four-element set $\{s, -s, p+s, p-s\}$ is uniquely determined by each of its elements. So, the set $S_4$ can be divided
into disjoint four-element sets, and (1) is proved.

(2) It follows from (1), that for an element $s \in S_4$ the elements $-s, p+s, p-s$ belong to $S_4$.
By the definition of a Cayley graph, the elements $s, -s, p+s, p-s$ are common neighbours of the vertices $0$ and $p$.
Thus, $S_4$ lies in the set of common neighbours of $0$ and $p$.

Now, let us take any element $s \in S_2$. By the definition of $S_2$, we have $p+s \not\in S$.
By the definition of a Cayley graph, we have $-(p+s) = p - s  \not\in S$. This means that $p$ is not adjacent to $s$, and
$s$ is not a common neighbour of the vertices $0$ and $p$. So, $0$ and $p$ have no common neighbours in $S_2$.
$\square$

\begin{lemma}\label{ComNeiOfSymVert}
For any vertex $y \in \mathbb{Z}_{2p}\setminus \{0,p\}$, the number $|N(-y, y)\setminus \{0,p\}|$ is even.
\end{lemma}
\proof
Let $i \in \mathbb{Z}_{2p}\setminus\{0,p\}$ be a common neighbour of the vertices $-y$ and $y$.
Then, by definition of Cayley graph, $i+y$ and $i-y$ belong to $S$.
Since $i+y = y - (-i)$ and $i-y = -y - (-i)$, each common neighbour $i$ of the vertices $y$ and $-y$, which is different from $0$ and $p$, gives one more common neighbour, namely $-i$.
This means that $|N(-y,y)\setminus \{0,p\}|$ is even.
$\square$

\begin{lemma}\label{ZeroAndElOfS2HaveOddNumOfComNei}
{\rm (1)} For a non-zero element $x \in \mathbb{Z}_{2p}$, the number $|N(0,x)|$ is odd if and only if
there exists an element $s \in S_2$ such that the equality $x = 2s$ holds (in other words, $x$ belongs to $2S_2$).

{\rm (2)} There are exactly $|S_2|$ vertices in $\Gamma$ having odd number of common neighbours with $0$.
\end{lemma}
\proof
(1) Let $x \in \mathbb{Z}_{2p}$ be a non-zero element such that $|N(0,x)|$ is odd.
By Lemma \ref{ElemsOfDifftParityHaveEvenNumOfComNeigh}, the element $x$ is even.
Let $y \in \mathbb{Z}_{2p}$ be an element such that $2y = x$ (note that there are only two such elements: $y$ and $y+p$). Then, by Lemma \ref{AutLemma}(3), the vertices $\psi_{1,-y}(0) = -y$ and $\psi_{1,-y}(x) = y$ have odd number of common neighbours.
By Lemma \ref{ComNeiOfSymVert}, the number $|N(-y,y)\setminus \{0,p\}|$ is even. Since $|N(-y,y)|$ is odd,
exactly one of the vertices $0$ and $p$ is a common
neighbour of the vertices $-y$ and $y$. Thus, exactly one of the vertices $y$ and $y+p$ belongs to $S$; denote this vertex by $s$.
Note, that $s \in S_2$ by the definition of $S_2$. Since $2y = 2(y+p) = x$, we have $x = 2s$.

Let us prove the converse. Note that, by Lemma \ref{AutLemma}(3), $|N(0,2s)| = |N(\psi_{1,-s}(0), N(\psi_{1,-s}(2s))| = |N(-s,s)|$ holds.
By Lemma \ref{ComNeiOfSymVert}, the number $|N(-s,s)\setminus \{0,p\}|$ is even. Since $s$ belongs to $S_2$,
we have $0 \in N(-s,s)$  and $p\not\in N(-s,s)$. Thus, $|N(-s,s)|$ is odd, and (1) is proved.

(2) It follows from (1), that exactly the vertices of the set $2S_2$ have odd number of common neighbours with the vertex $0$.
Note that, for any $i_1, i_2 \in \mathbb{Z}_{2p}$, the equality $2i_1 = 2i_2$ holds if and only if $i_2 \in \{i_1,i_1+p\}$.
By definition of $S_2$, we obtain the equality $|2S_2| = |S_2|$. $\square$

\begin{lemma}\label{S2IsEmpty}
The set $S_2$ is empty.
\end{lemma}
\proof Suppose to the contrary that $S_2$ is non-empty. By Lemma \ref{S4IsSetOfComNei0andp}(1), $|N(0,p)|$ is divided by $4$.
Put $c_0:=|N(0,p)| = |S_4|$. By Lemma \ref{ZeroAndElOfS2HaveOddNumOfComNei},
there exists at least one pair of vertices that have an odd number of common neighbours.
Denote this odd number by $c_1$. Thus, the equality $\{c_0, c_1\} = \{a,b\}$ holds.

Denote by $S'$ and $S''$ the sets of odd and even elements of $S$, respectively.
Let us count in two ways the number of paths of length $2$,
connecting $0$ with some odd vertex. On the one hand, this equals $pc_1$. On the other hand, each pair of an odd element $s_1 \in S$ and
an even element $s_2 \in S$ gives the two paths $0\sim s_1 \sim (s_1+s_2)$ and $0\sim s_2 \sim (s_1+s_2)$,
connecting $0$ with the odd vertex $s_1+s_2$.
Thus, the number of the paths of length $2$ is equal to the doubled number of the pairs of an odd element $s_1 \in S$ and
an even element $s_2 \in S$.
The number of such pairs is equal to $|S'||S''|$. We obtain the following equality
$$2|S'||S''| = pc_0.$$

So, $p$ divides $2|S'||S''|$. Since $|S''|$ is even and $|S''| < 2p$ holds,
we conclude that $p$ divides $|S'|$ (otherwise, $|S''| = 0$, which means that $\Gamma$ is bipartite and cannot be a strictly Deza graph). Consequently,
$|S'| \in \{0,p\}$ holds. If $|S'| = 0$, then $\Gamma$ is disconnected, a contradiction. Thus, we have $|S'| = p$
and, consequently, $|S''| = c_0/2$.
The equality $|S'| = p$ means that each odd element of $\mathbb{Z}_{2p}$, including $p$, belongs to $S$.
So, the elements of $S_2$ are odd.

Since $S = S' \cup S''$ holds, we obtain
$$k = |S| = |S'| + |S''| = p + c_0/2.$$

Since $S = S_2 \cup S_4 \cup \{p\}$ holds, we obtain
$$|S_2| = |S| - |S_4| - 1 = k - c_0 - 1 = p - c_0/2 - 1.$$

By Lemma \ref{ZeroAndElOfS2HaveOddNumOfComNei}(2), there are exactly $|S_2|$ vertices in $\Gamma$,
having $c_1$ common neighbours with $0$. Thus, there are exactly $2p - |S_2| - 1$ vertices, having
$c_0$ common neighbours with $0$.

Now let us count in two ways the number of paths of length $2$ with initial vertex is $0$.
On one hand, this number equals $k(k-1)$. On the other hand, each vertex that has $c_0$ (resp. $c_1$) common neighbours with $0$
gives $c_0$ (resp. $c_1$) such paths. So, there are exactly $|S_2|c_1 +(2p - |S_2| - 1)c_0$ paths of length 2.
We obtain the equality
$$k(k-1) = |S_2|c_1 +(2p - |S_2| - 1)c_0$$
or, equivalently,
$$(p + c_0/2)(p + c_0/2-1) = (p - c_0/2 - 1)c_1 +(2p - (p - c_0/2 - 1) - 1)c_0.$$

Let us make some calculations.

$$p^2 + pc_0 + c_0^2/4 - p - c_0/2 = (p - c_0/2 - 1)c_1 +(p + c_0/2)c_0,$$
$$p^2  - c_0^2/4 - p - c_0/2 = (p - c_0/2 - 1)c_1,$$
$$(p  - c_0/2)(p+c_0/2) - (p + c_0/2) = (p - c_0/2 - 1)c_1,$$
$$(p+c_0/2)(p  - c_0/2 - 1) = (p - c_0/2 - 1)c_1.$$

If $p  - c_0/2 - 1 = 0$ holds, then $c_0 = 2p - 2$ and, consequently, $S = \mathbb{Z}_{2p}\setminus\{0\}$ holds. A contradiction.
Thus, we have
$$c_1 = p + c_0/2 = k.$$

Now the relation defined by the rule ``to coincide or have $c_1 = k$ common neighbours'' is an equivalence relation.
So, the vertex set of $\Gamma$ can be divided into equinumerous equivalence classes of cardinality $|S_2| + 1 = p - c_0/2$.
Thus, $p - c_0/2$ divides $2p$. Since $p - c_0/2$ is odd, we conclude that $p - c_0/2 = p$ and, consequently, $c_0 = 0$ holds.
This means that $S_4$ is empty. So, $S$ consists of all odd elements, and $\Gamma$ is the complete bipartite graph. A contradiction.
$\square$

\medskip

It follows from Lemma \ref{ZeroAndElOfS2HaveOddNumOfComNei}(2) and Lemma \ref{S2IsEmpty} that the parameters
$a$,$b$ of $\Gamma$ are even.
So, there are two possible cases: either $p \in S$ and $b = |S_4| = k-1$ holds or $p \not\in S$
and $b = |S_4| = k$. In particular, the inequality $b \ge k-1$ holds.

Note that $S$ contains at least one odd vertex (otherwise, the graph $\Gamma$ is disconnected).

\begin{lemma}\label{OddVertInSHaveaComNeiWith0}
For any odd vertex $x \in S_4$, the equality $|N(0,x)| = a$ holds.
\end{lemma}
\proof
Suppose to the contrary that there exists an odd vertex $x \in S_4$ such that $|N(0,x)| = b$ holds.
Since $0$ and $x$ are adjacent, we have $k-1  \le b = |N(0,x)| \le |N(0)\setminus\{x\}| = k-1$. So, the equality $b = k-1$ holds.
This implies that the vertices $0$ and $p$ have $k-1$ common neighbours, so, $p$ belongs to $S$.

Consider a vertex $y$ that is adjacent to $0$,  $y \ne -x$. Since $|N(0,x)| = b = k - 1$ holds, in view of Lemma \ref{AutLemma}(3), we have
$|N(y,x+y)| = |N(\psi_{1,y}(0),  \psi_{1,y}(x))| = |N(0,x)| = b = k-1$.
Since $N(y)\setminus\{x+y\} = N(x+y)\setminus\{y\}$ holds and $0 \in N(y)\setminus\{x+y\}$,
we conclude that $x+y$ is adjacent to $0$. We then consider $x+y$ instead of $y$ and conclude that $2x+y$ is adjacent to $0$. Let us put $y:=x$ and apply induction to $y$. Since $x$ and $2p$ are coprime, the element $x$ generates $\mathbb{Z}_{2p}$, which gives that $0$ is adjacent to $x, 2x, 3x, \ldots, (2p-1)x$.
We obtain that each vertex of $\Gamma$ is adjacent to $0$.
This means that $\Gamma$ is a clique. A contradiction. $\square$

\begin{lemma}\label{OddVertInSecNeiOf0HaveaComNeiWith0}
For any odd vertex $x \in N_2(0)$, $x\ne p$, the equality $|N(0,x)| = a$ holds.
\end{lemma}
\proof
Suppose to the contrary that there exists an odd vertex $x \in N_2(x)$ such that $x\ne p$ and $|N(0,x)| = b$ hold.
In particular, the inequality $\beta > 1$ holds.

Let us consider the following two cases.

Let $b = k$ holds. Since $b$ is even, $k$ is also even, which means that $p$ does not belong to $S$. Now, the relation defined on the vertex set of $\Gamma$ by the rule
``to coincide or have $b$ common neighbours'' is an equivalence relation; the vertex set of $\Gamma$ can be divided into
equinumerous equivalence classes of cardinality $\beta + 1$. So, $\beta+1$ divides $2p$.
Note, that, for any vertex $y \in N_2(x)$ such that $y\ne p$ and $|N(0,y)| = b$, the equalities $|N(0,-y)| = |N(\psi_{-1,0}(0), \psi_{-1,0}(-y))|
= |N(0,y)| = b$ hold; so, the number of the such vertices $y$ is even. Since $|N(0,p)| = b$ holds, we conclude that $\beta$ is odd.
Thus, the fact that $\beta+1$ divides $2p$ implies $\beta + 1 = 2$ and, consequently, $\beta = 1$. A contradiction.

Let $b = k-1$ holds. Then $p$ belongs to $S$. Recall that $N(0)\setminus\{p\} = N(p)\setminus\{0\} = S_4$.
This means that any vertex $x\in V(\Gamma)\setminus\{0,p\}$ either is adjacent to $0$ and $p$ or is not adjacent to both of them.
Now, the relation defined on the vertex set of $\Gamma$ by the rule
``to coincide or have $b$ common neighbours'' is an equivalence relation; the vertex set of $\Gamma$ can be divided into
equinumerous equivalence classes of order $\beta + 1$. So, $\beta+1$ divides $2p$. Since $\beta$ is odd, we conclude that $\beta+1 = 2$
and, consequently,  $\beta = 1$ hold.
A contradiction. $\square$

\begin{lemma}\label{betaeq1}
For the graph $\Gamma$, $\beta = 1$ holds.
\end{lemma}
\proof
By Lemma \ref{OddVertInSHaveaComNeiWith0} and Lemma \ref{OddVertInSecNeiOf0HaveaComNeiWith0}, each odd vertex,
excepting $p$,
has $a$ common neighbours with $0$.
It is enough to prove that each non-zero even vertex has $a$ common neighbours with $0$.
Note that each non-zero even vertex can be represented as $x+p$, where $x$ is odd and $x\ne p$ holds.
Since $S_2$ is empty, the vertex $0$ is either adjacent to $x$ and $x+p$ or not adjacent to both of them.
Since the equality $N(x)\setminus\{x+p\} = N(x+p)\setminus\{x\}$ holds,
we conclude that $N(0,x)\setminus\{x+p\} = N(0,x+p)\setminus{x}$ holds.
Since $x$ is odd, this means that $|N(0,x+p)| = |N(0,x)| = a$. We have proved that all even vertices, excepting 0, have $a$ common neighbours with $0$.
Thus, in view of Lemma \ref{OddVertInSHaveaComNeiWith0} and Lemma \ref{OddVertInSecNeiOf0HaveaComNeiWith0}, the only vertex that has $b$ common neighbours with the vertex $0$ is the vertex $p$, and we obtain the equality $\beta = 1$.
$\square$

\medskip

Note that, in view of Lemma \ref{AutLemma}(3), for any $x \in V(\Gamma)$, we have $|N(x,x+p)| = |N(\psi_{1,x}(0),  \psi_{1,x}(p))| = |N(0,p)| = b$.
Since $S_2$ is empty, there are either
all possible edges between the sets $\{x_1,x_1+p\}$ and $\{x_2,x_2+p\}$ or no such edges,
for any two vertices $x_1,x_2 \in \mathbb{Z}_{2p}$, $x_1 \ne x_2+p$.

Denote by $V'$ the set $\{\{x,x+p\}~|~x \in \mathbb{Z}_{2p}\}$ and consider the graph $\Gamma'$ with the vertex set $V'$
and the following adjacency rule:
for any two elements $x_1,x_2 \in \mathbb{Z}_{2p}$, $x_1 \ne x_2+p$, the vertices $\{x_1,x_1+p\}$ and $\{x_2,x_2+p\}$ are adjacent in $\Gamma'$
if and only if there exists an edge between
the sets $\{x_1,x_1+p\}$ and $\{x_2,x_2+p\}$ in the graph $\Gamma$.

\begin{lemma}\label{GammaPrimeIsSRG}
The following statements hold.

{\rm (1)} If $p \in S$, then the graph $\Gamma'$ is strongly regular with parameters
$(n',k',\lambda',\mu') = (p, \frac{k-1}{2},\frac{a-2}{2}, \frac{a}{2})$;

{\rm (2)} If $p \not\in S$, then the graph $\Gamma'$ is strongly regular with parameters
$(n',k',\lambda',\mu') = (p, \frac{k}{2},\frac{a}{2}, \frac{a}{2})$.

\end{lemma}
\proof
(1) If $p \in S$, then $b = k - 1$ holds. Let us count the number of neighbours of a vertex $\{x,x+p\}\in V(\Gamma')$.
Note, that, for the vertex $x \in \Gamma$, the equality $N(x) = \{x+p\} \cup (x+S_4)$ holds.
By definition of $S_4$, any element $y \in x+S_4$ lies in $x+S_4$ with the element $y+p$;
each such a pair $\{y, y+p\}$ represents a vertex in the graph $\Gamma'$.
So, the vertex $\{x,x+p\}$ has $(k-1)/2$ neighbours in $\Gamma'$.

Let us consider an arbitrary pair of adjacent vertices $\{x,x+p\}$ and $\{y,y+p\}$ in the graph $\Gamma'$.
By Lemma \ref{betaeq1}, the vertices $x$ and $y$ have $a$ common neighbours in $\Gamma$.
Note, that the inclusion $\{x+p,y+p\} \subseteq N(x,y)$ holds. The elements of
$N(x,y)\setminus \{x+p,y+p\}$ can be divided into disjoint pairs with the difference $p$;
these pairs represent $|N(x,y)\setminus \{x+p,y+p\}|/2 = (a-2)/2$ common neighbours of the vertices
$\{x,x+p\}$ and $\{y,y+p\}$ in the graph $\Gamma'$.

Let us consider an arbitrary pair of non-adjacent vertices $\{x,x+p\}$ and $\{y,y+p\}$ in the graph $\Gamma'$.
By Lemma \ref{betaeq1}, the vertices $x$ and $y$ have $a$ common neighbours in $\Gamma$.
The elements of $N(x,y)$ can be divided into disjoint pairs with the difference $p$;
these pairs represent $|N(x,y)|/2 = a/2$ common neighbours of the vertices
$\{x,x+p\}$ and $\{y,y+p\}$ in the graph $\Gamma'$.

(2) If $p \not\in S$, then $b = k$ holds. Let us count the number of neighbours of a vertex $\{x,x+p\}\in V(\Gamma')$.
Note that, for the vertex $x \in \Gamma$, the equality $N(x) = x+S_4$ holds.
By definition of $S_4$, any element $y \in x+S_4$ lies in $x+S_4$ with the element $y+p$;
each such a pair $\{y, y+p\}$ represents a vertex in the graph $\Gamma'$.
So, the vertex $\{x,x+p\}$ has $k/2$ neighbours in $\Gamma'$.

Let us consider an arbitrary pair of distinct vertices $\{x,x+p\}$ and $\{y,y+p\}$ in the graph $\Gamma'$.
By Lemma \ref{betaeq1}, the vertices $x$ and $y$ have $a$ common neighbours in $\Gamma$.
The elements of $N(x,y)$ can be divided into disjoint pairs with the difference $p$;
these pairs represent $|N(x,y)|/2 = a/2$ common neighbours of the vertices
$\{x,x+p\}$ and $\{y,y+p\}$ in the graph $\Gamma'$. $\square$

\begin{lemma}\label{GammaPrimeIsACirculant}
The graph $\Gamma'$ is a circulant.
\end{lemma}
\proof
Let us consider $V'$ as the quotient group $\mathbb{Z}_{2p}/\{0,p\}$.
Note that the group $V'$ is isomorphic to $\mathbb{Z}_{p}$.
Denote by $\hat{S_4}$ the set $\{\{s,s+p\}~|~s \in S_4\}$. Since the equality $-S_4 = S_4$ holds in the group $\mathbb{Z}_{2p}$,
the equality $-\hat{S_4} = \hat{S_4}$ holds in the quotient group.
Since any vertices $\{x,x+p\}, \{y,y+p\} \in V'$ are adjacent in $\Gamma'$ if and only if
$\{x,x+p\} - \{y,y+p\} = \{x-y,x-y+p\} \in \hat{S_4}$,
we have $\Gamma' = Cay(V', \hat{S_4})$. $\square$

\medskip
Now we are ready to complete the proof of Theorem \ref{circ2p}.

By Lemma \ref{GammaPrimeIsSRG} and Lemma \ref{GammaPrimeIsACirculant}, the graph $\Gamma'$ is a strongly regular circulant.
By Theorem \ref{SRGCirculant}, the graph $\Gamma'$ is isomorphic to $P(p)$ and, in particular, $p \equiv 1(4)$ holds.
Thus, $k'$ is even and, consequently, $p$ belongs to $S$. We obtain, that $\Gamma$ is isomorphic to 2-clique-extension
of $P(p)$. $\square$
\end{proof}

\section*{Acknowledgment} \label{Ack}
The reported study was funded by RFBR according to the research project 20-51-53023.
 The work is partially supported by Mathematical Center in Akademgorodok, the
agreement with Ministry of Science and High Education of the Russian
Federation number 075-15-2019-1613.
The authors thank Alexander Gavrilyuk who took part in discussions concerning the result in Section 2.1.

\end{document}